\documentclass{amsart}

\usepackage[]{amsthm,amsmath,amssymb,amstext}
\allowdisplaybreaks

\usepackage[T1]{fontenc}
\usepackage{txfonts}

\usepackage[english]{babel}
\usepackage{hyperref}
\usepackage{tabularx}
\usepackage[shortlabels]{enumitem}
\setlist[enumerate]{label=\alph*)}
\usepackage{mathrsfs}
\usepackage{youngtab}
\usepackage{ytableau}

\usepackage{tikz-cd}

\newtheorem{z}{asdfasdf}[section]
\newtheorem{Definition}[z]{Definition}
\newtheorem{definition}[z]{Definition}
\newtheorem{Lemma}[z]{Lemma}

\newtheorem{Corollary}[z]{Corollary}

\newtheorem{Remark}[z]{Remark}
\newtheorem{remark}[z]{Remark}
\newtheorem{Proposition}[z]{Proposition}

\newtheorem{Theorem}[z]{Theorem}

\renewcommand{\O}{\mathcal{O}}
\DeclareMathOperator{\N}{\mathbb{N}}
\DeclareMathOperator{\Z}{\mathbb{Z}}

\let\sl\relax
\DeclareMathOperator{\s}{\mathbf{s}}
\DeclareMathOperator{\sl}{\widehat{\mathfrak{sl}}}
\DeclareMathOperator{\PP}{\mathcal{P}}
\DeclareMathOperator{\CC}{\mathcal{C}}

\DeclareMathOperator{\C}{\mathbb{C}}
\DeclareMathOperator{\sy}{\mathfrak{S}}

\DeclareMathOperator{\ind}{Ind}
\DeclareMathOperator{\res}{Res}
\DeclareMathOperator{\eres}{res}

\DeclareMathOperator{\irr}{Irr}
\DeclareMathOperator{\iind}{\mathit{i-}Ind}

\DeclareMathOperator{\rad}{rad}

\DeclareMathOperator{\End}{End}

\DeclareMathOperator{\KZ}{KZ}
\DeclareMathOperator{\soc}{soc}
\DeclareMathOperator{\head}{head}
\DeclareMathOperator{\QQ}{{\textbf{Q}}}
\DeclareMathOperator{\HH}{\mathcal{H}}
\DeclareMathOperator{\hh}{\mathbf{H}}
\DeclareMathOperator{\hhh}{\mathbf{h}}
\DeclareMathOperator{\modcat}{-mod}

\DeclareMathOperator{\vareps}{{\varepsilon}}
\DeclareMathOperator{\Oint}{\mathcal{O}_{\text{int}}}
\newcommand{\tei}{\widetilde{e_i}}
\newcommand{\tFi}{\widetilde{F_i}}
\newcommand{\tEi}{\widetilde{E_i}}
\newcommand{\tfi}{\widetilde{f_i}}

\begin{document}
\sloppy
\title{On the number of constituents of induced modules of Ariki-Koike algebras}

\address{RWTH Aachen University, Lehrstuhl D f\"ur Mathematik, Pontdriesch 14-16, 52062 Aachen, Germany}

%

\author{Christoph Schoennenbeck\\RWTH Aachen University}
\email{christoph.schoennenbeck@rwth-aachen.de}
\thanks{Research supported by the German Research Foundation (DFG) research training group \textit{Experimental and constructive algebra} (GRK 1632).\\ }

\begin{abstract}
 We examine the crystal graph of the $\sl_e$-module arising from an $\sl_e$-categorification to study the defining endo-functors of the categorification. 
 This yields lower bounds on the number of irreducible constituents of certain objects.
 We use Ariki's categorification result on Ariki-Koike algebras to obtain a new lower bound on the number of constituents of their parabolically induced modules.
 In particular this will imply reducibility of every induced module.
\end{abstract}

\maketitle

\section*{Introduction}
Denote by $\sl_e$ the affine Lie algebra of type $A_{e-1}^{(1)}$, by $U(\sl_e)$ its enveloping algebra and by $U_v(\sl_e)$ its quantum enveloping algebra.
Much of the structure of certain classes of $U_v(\sl_e)$- and $U(\sl_e)$-modules can be encoded in so-called crystal graphs via the concepts of crystal bases and perfect bases, respectively, cf. e.g. \cite{HongKang, BerKaz}. 
These crystal graphs have a nice combinatorial description stemming from the realisation of irreducible highest weight modules as submodules of Fock spaces cf. \cite{JMMO, ArikiBook}.
We will exploit this combinatorial description to study categories possessing a so-called $\sl_e$-categorification of a $\C$-linear abelian category $\CC$, as defined in \cite{Rouquier}.
One key ingredient to such a categorification is a pair of adjoint endo-functors $(U,V)$ of $\CC$ which decompose as direct sums of $e$ summands, and this decomposition yields an $\sl_e$-module structure on the complexification of the Grothendieck group of $\CC$, i.e. on $\C\otimes_{\Z} R_0(\CC)$. \\
If this $\sl_e$-module is an element of the so-called category $\Oint$ of $\sl_e$-modules, we combine results by Shan and Chuang-Rouquier, cf. \cite{Shan, ChuangRouquier}, with a combinatorial observation to obtain a new lower bound on the number of constituents of images under $V$.\\
These results are applicable to a number of settings, in particular to category $\O$ of rational cyclotomic Cherednik algebras and the representations of Ariki-Koike algebras, cf. \cite{Shan,ArikiBook}.
Both are closely related to complex reflection groups of type $G(r,1,n)$, that is, groups of the form $\left(\Z/r\Z\right)\wr \sy_n$, where $\sy_n$ is the symmetric group on $n$ letters.
In both cases, the functor $V$ is given by parabolic induction, thus we can use our general result to study lower bounds on the number of constituents of parabolically induced modules.

Motivated by this result on parabolic induction for certain Ariki-Koike algebras we go on to prove analogous bounds for Ariki-Koike algebras with arbitrary invertible parameters, as not all parameter choices are covered by the $\sl_e$-categorification result: 
If $K$ is a field, then the Ariki-Koike algebra over $K$ is defined via generators and relations involving parameters $q,Q_1,\dots, Q_r\in K^*$.
The $\sl_e$-categorification result is known to hold in the case that $q\neq 1 $ is a root of unity of finite order and the parameters $Q_1,\dots, Q_r$ are so-called \emph{$q$-connected}, cf. \cite[Thm 12.5]{ArikiBook}. 
Hence, to obtain a complete result we first reduce the task of computing the number of constituents of induced modules to $q$-connected parameters via the Morita equivalence result by Dipper-Mathas, cf. \cite{DipperMathas}.\\
Then it remains to consider the cases that $q$ is either $1$ or has infinite order in $K^*$.
While the latter case is handled quite similarly to our study of $\sl_e$-categorification, the former requires some hands-on computation.\\
The main result on parabolic induction of Ariki-Koike algebras with arbitrary invertible parameters is then given in Theorem \ref{Main_Theorem_Hecke}.\\

This article is structured as follows:\\
We establish the necessary vocabulary for the representation theory of $\sl_e$, in particular integrable modules, category $\Oint$,  perfect bases, and crystal graphs.
Then we introduce certain crystal graphs, study their combinatorics, and indicate how to obtain the crystal graphs of all elements of $\Oint$ from the ones we defined. 
After recalling the definition of $\sl_e$-categorification we present our main result on general $\sl_e$-categorification in Theorem \ref{Thm_ConstituentsCategorification}.\\
As an application we consider parabolic induction in rational cyclotomic Cherednik algebras, cf. Theorem \ref{Theo_MainTheoremCherednik}.

The second chapter is concerned with the study of Ariki-Koike algebras and their parabolic induction. 
To obtain the desired lower bound we reduce the task to the case of$q$-connected parameter sets and handle the cases not covered by Ariki's $\sl_n$-categorification result separately.\\
We close by proving an analogue of Theorem \ref{Main_Theorem_Hecke} for the closely related degenerate cyclotomic Hecke algebras, cf. Theorem \ref{Thm_MainTheoremDegenerate}.

\section{\texorpdfstring{$\sl_e$-categorification}{\hat{sl_n}-categorification} and crystal graphs}
\subsection{The Kac-Moody algebra and crystal graphs}
 We start of by defining the affine Lie algebra $\sl_e$, i.e. the Kac-Moody algebra of type $A^{1}_{e-1}$ following \cite{HongKang}.
Let $e\geq 2$ be an integer. We give the following definitions only for $e\geq 3$, but for $e=2$ the construction is similiar, cf. \cite{Kac} for details.\\
Let $\mathfrak{h}$ be a $\C$-vector space with basis $\{h_1,\dots, h_{e-1}, d\}$ and $\{\Lambda_0,\dots, \Lambda_{e-1}, \partial\}$ a $\C$-basis of $\mathfrak{h}^*$ such that 
\[
 \Lambda_i(h_j)=\delta_{i,j}, \quad \Lambda_i(d)=\partial(h_i)=0, \quad \partial(d)=1, 
\]
for $0\leq i,j\leq e-1$. 
For ease of notation we set $\Lambda_z:=\Lambda_{z\pmod{e}}$ for any integer $z$.\\
For $0\leq i\leq e-1$ we define further elements of $\mathfrak{h}^*$ by 
\[
 \alpha_i:=-\Lambda_{i-1}+2\Lambda_i-\Lambda_{i+1}+\delta_{0,i}\partial.
\]

 The \emph{affine Lie algebra $\sl_e$} is the Lie algebra generated by the elements $e_i$, $f_i$ for $0\leq i\leq e-1$ and $\{h_0,\dots, h_{e-1},d\}$ subject to the following relations:
 \begin{gather*}
  [h,e_i]=\alpha_i(h)e_i,\\
  [h, f_i]=-\alpha_i(h)f_i,\\
  [e_i, f_j]=\delta_{i,j} h_i,\quad [h, h']=0,\\
  [e_i, [e_i, e_j]]=[f_i,[f_i,f_j]]=0, \ \text{if } (i-j)\equiv \pm1\pmod{e},\\
  [e_i,e_j]=[f_i,f_j]=0, \ \text{if } (i-j)\not\equiv \pm1\pmod{e},
 \end{gather*}
 for $h, h' \in \mathfrak{h}$ and $0\leq i,j\leq e-1$.\\
Its derived subalgebra ${\sl_e}^\prime=[\sl_e,\sl_e]$ is spanned by the elements $e_i$, $f_i$, and $h_i$ for $0\leq i\leq e-1$.

We call the $\Lambda_i$ the \emph{fundamental weights} of $\sl_e$ and $\partial$ the \emph{null root}.
Furthermore, the $\alpha_i$ are known as \emph{simple roots} and the $h_i$ as \emph{simple co-roots}
We define the \emph{weight lattice} $P:=\Z \partial\oplus\bigoplus_{i=0}^{e-1} \Z \Lambda_i$ and the \emph{dominant integral weights} $P^+:=\Z\partial\oplus \bigoplus_{i=0}^{e-1}\Z_{\geq 0}\Lambda_i$. 
Finally, we set $\overline{P^+}:=\bigoplus_{i=0}^{e-1}\Z_{\geq 0}\Lambda_i$, the\emph{ classical dominant integral weights.}

In the following we will be concerned with certain representations of $\sl_e$ or, equivalently, of $U(\sl_e)$, its universal enveloping algebra.
All modules studied here will have a \emph{weight space decomposition}:
For an $\sl_e$-module $M$ and some $\lambda\in \mathfrak{h}^*$ denote by $M_\lambda:=\{m\in M\mid hm=\lambda(h)m\text{ for all } h\in \mathfrak{h}\}$ the \emph{weight space of $M$ of weight $\lambda$}.

An $\sl_e$-representation is called \emph{integrable} if the Chevalley generators $e_i$ and $f_i$ for $0\leq i\leq e-1$ of $\sl_e$ act locally nilpotently.
We say that an $\sl_e$-module $M$ is in \emph{category $\Oint$} if
\begin{itemize}
\item $M$ is integrable, 
 \item $M$ has a weight space decomposition $M=\oplus_{\lambda}M_{\lambda}$ and $M_\lambda$ is finite dimensional for all $\lambda$,
 \item there exists a finite set $F\subseteq P$ such that $\text{wt}(M)\subseteq F+\sum_{j=0}^{e-1} \Z_{\leq 0}\alpha_i$, where $\text{wt}(M)$ is the set of weights $\lambda$ in $ P$ such that $M_\lambda\neq 0$.
\end{itemize}

If $M$ is in $\Oint$, then $M$ decomposes as a direct sum of \emph{irreducible highest weight modules} $L(\lambda)$ with weight $\lambda$, where $\lambda$ is in $P^+$, and every irreducible weight module $L(\lambda)$ with $\lambda $ in $P^+$ is an element of $\Oint$.\\
Every module $M$ in $\Oint$ has a \emph{perfect basis} in the sense of \cite{BerKaz}, i.e. a basis $B$ consisting of weight vectors equipped with functions $\tEi,\tFi: B\to B\dot{\cup} \{0\}$ for $0\leq i\leq e-1$ such that 
\begin{itemize}
 \item for $b, b'$ in $B$ it is $\tFi(b)=b'$ if and only if $\tEi (b')=b$,
 \item It is $\tEi(b)\neq 0$ if and only if $ e_ib\neq 0$, where $e_0, \dots, e_{n-1}$ and $f_0, \dots, f_{n-1}$ are again the Chevalley generators of $\sl_e$, 
 \item if $e_i b\neq 0$, then 
 \[
  e_i b \in \C^* \tEi(b)+V_i^{<\ell_i(b) -1}, 
 \]
 where $\ell_i(v):=\max\{j\geq 0\mid e_i^j v\neq 0\}$ and $V_i^{<k}:=\{v\in M\mid \ell_i(v)<k\}$.
\end{itemize}

To a perfect basis of $M$ we can associate an abstract crystal in the sense of \cite[Definition 4.5.1]{HongKang}. 
However, we will only be interested in its \emph{crystal graph}. 
If $M$ is in $\Oint$ with a crystal basis $B$, then the crystal graph associated to $B$ is a directed graph with coloured edges, whose vertex set is $B$ and for $b, b'$ in $B$ there is an edge $b\stackrel{i}{\rightarrow}b'$ with label $i$ if and only if $\tFi(b)=b'$.
\begin{definition}
 A \emph{crystal graph isomorphism} is an isomorphism of coloured graphs between crystal graphs of perfect bases, i.e. if $B$ and $C$ are perfect bases of modules $M$ and $N$, then a crystal isomorphism is a bijection $\phi:B\to C$ such that there is an edge $b\stackrel{i}{\rightarrow}b'$ in the crystal graph associated to $B$ if and only if there is an edge $\phi(b)\stackrel{i}{\rightarrow }\phi(b')$ in the crystal graph associated to $C$. 
\end{definition}
For modules in $\Oint$ there is only one associated crystal graph.
\begin{Lemma}
If $B$ and $B'$ are two perfect bases of $M\in \Oint$, then the crystal graphs associated to $B$ and $B'$ are isomorphic.
Thus, it makes sense to speak of the crystal graph associated to $M$.
\end{Lemma}
\begin{proof}
 This follows from \cite[Main Thm 5.37]{BerKaz} just as in the proof of \cite[Thm 6.3]{Shan}.
\end{proof}

The crystal graph associated to an element of $\Oint$ only depends on its ${\sl_e}^\prime$-structure.
\begin{Lemma}\label{Lemma_RestrictedModuleIsoGraphIso}
 Let $M, N\in \Oint$. 
 If $M$ and $N$ are isomorphic as ${\sl_e}^\prime$-modules, then their associated crystal graphs are isomorphic.
\end{Lemma}
\begin{proof}
Let $\psi:M\to N$ be an ${\sl_e}^\prime$-module isomorphism and $\left(B, \tEi, \tFi\right)$ a perfect basis of $M$. 
As $\psi$ is compatible with the action of the $e_i$, it follows that $\psi(B)$ together with $\psi\circ \tEi\circ \psi^{-1}$ and $\psi\circ \tFi\circ \psi^{-1}$ constitutes a perfect basis of $N$. 
Hence, the restriction $\psi:B\to \psi(B)$ is a crystal graph isomorphism.
\end{proof}
\begin{Corollary}\label{Cor_GraphHighestWeightMod}
 Let $\lambda, \lambda'\in P^+$.
 If $\lambda\cong \lambda'\pmod{\Z\partial}$, then the crystal graphs associated to $L(\lambda)$ and $L(\lambda')$ are isomorphic.
\end{Corollary}
\begin{proof}
 As ${\sl_e}^\prime$-modules, $L(\lambda)$ and $L(\lambda')$ are isomorphic, cf. \cite[\S 9.10]{Kac}.
\end{proof}

In the following we describe the crystal graph associated to an irreducible highest weight module $L(\lambda)$ for $\lambda\in P^+$. 
Since elements of $\Oint$ are direct sums of such highest weight modules, this will provide a description of the crystal graphs associated to all elements of $\Oint$.\\
We fix yet more notation.\\
For a positive integer $r$ and a tuple $\mathbf{t}=(t_1,\dots, t_r)\in\Z^{r}$ we define $\Lambda_{\mathbf{t}}:=\Lambda_{t_1}+\dots+\Lambda_{t_r}$.\\
For $\lambda=\Lambda_{s_1}+\dots+\Lambda_{s_r}+k\partial\in P^+$ we define $\s_\lambda$ to be the unique element of $\widetilde{\Z_{\geq 0}^r}:=\left\{(t_1,\dots, t_r)\mid 0\leq t_1\leq \dots\leq t_r<e \right\}$ such that $\Lambda_{\s_{\lambda}}\equiv \lambda\pmod{\Z\partial}$.
Clearly, $\s_\lambda$ is well-defined, since the sum $\Lambda_{s_1}+\dots+\Lambda_{s_r}$ is uniquely defined by the multiset $\{s_1\pmod{e},\dots, s_r\pmod{e}\}$.
\\
By Corollary \ref{Cor_GraphHighestWeightMod}, the crystal graph associated to $L(\lambda)$ only depends on $\Lambda_{\s_\lambda}$ up to crystal graph isomorphism.
Thus, until the end of this section we fix an integer $r\geq 1$ and some $\s\in \widetilde{\Z_{\geq 0}^r}$. \\

The description of the crystal graph associated to $L(\lambda)$ has been achieved via Fock space theory. 
To state it we need some combinatorial ground work.  \\
Let $n$ be a non-negative integer. A \emph{partition of $n$} is a tuple $\alpha=(\alpha_1,\dots, \alpha_\ell)$ of non-increasing non-negative integers $\alpha_1,\dots, \alpha_\ell$ such that $|\alpha|:=\sum_{i=1}^\ell \alpha_i=n$.
We write $\alpha\vdash n$ if $\alpha$ is a partition of $n$.
An \emph{$r$-multipartition} of $n$ is a tuple $\lambda=(\mu^{(1)},\dots, \lambda^{(r)})$ where each $\lambda^{(i)}$ is a partition and $|\lambda|:=\sum_{i=1}^r |\lambda^{(i)}|=n$. 
We write $\lambda\vdash_r n$ if $\lambda$ is an $r$-multipartition of $n$.
The \emph{Young diagram $[\lambda]$} of an $r$-multipartition $\lambda\vdash_r n$ is the set 
\[
 \left\{(a,b,c)\mid 1\leq \lambda^{(c)}_a\leq b, \ 1\leq c\leq r\right\}.
\]
The elements of $[\lambda]$ are called \emph{nodes}. 
More generally, we call any element of $\N\times\N\times \{1,\dots, r\}$ a node. 
A node $x$ is called an \emph{addable node of $\lambda$} if $x\notin [\lambda]$ and $[\lambda]\cup \{x\}$ is the Young diagram of an $r$-multipartition of $n+1$. 
We write $\lambda\cup \{x\}$ for the corresponding $r$-multipartition.\\
Similarly, $x$ is called a \emph{removable node of $\lambda$} if $x\in [\lambda]$ and $[\lambda]\setminus \{x\}$ is the Young diagram of an $r$-multipartition of $n-1$.
We write $\lambda\setminus \{x\}$ for the corresponding $r$-multipartition.
Finally, if a multipartition $\mu$ is obtained from $\lambda$ by adding exactly one node, say $x$, i.e. $\mu=\lambda\cup \{x\}$, we write $\lambda\setminus \mu=x$.

 For two nodes $x=(a,b,c)$ and $y=(a',b',c')$ of $\lambda$ we say that \emph{$x$ lies above or higher than $y$} if $c<c'$ or $c=c'$ and $a<a'$ or $c=c', a=a'$ and $b>b'$. 
 We also say that $y$ lies below or lower than $x$.\\
 This is the abstract notion of the usual visual way of writing down the diagram of $\lambda$ by depicting the diagrams of the components below one another, starting with $\lambda^{(1)}$.
  This yields a total order on the Young diagram $[\lambda]$.

\begin{definition}
 Let $\lambda\vdash_r n$.
 The \emph{residue of a node $x=(a,b,c)\in [\lambda]$} (with respect to $\s$) is defined as $\eres(x):= b-a +s_c \pmod{e}$.
 If $\eres(x)\equiv i\pmod{e}$ for $0\leq i\leq e-1$, then we call $x$ an $i$-node. \\
 Addable $i$-nodes are called $i$-addable and removable $i$-nodes are called $i$-removable.\\
\end{definition}

We follow \cite{ArikiBook} to define a number of different objects to construct a certain crystal graph.\\
We define normal, co-normal, good, and co-good nodes of $\lambda\vdash_r n$:\\
Choose some $1\leq i\leq e-1$ and write down the sequence of addable and removable $i$-nodes sorted from highest to lowest.
Encode every addable node with the symbol $+_i$ and every removable one with the symbol $-_i$. 
The resulting sequence is called the \emph{$i$-signature of $\lambda$}.\\
Now recursively remove all pairs $-_i+_i$ from this sequence until this is no longer possible to finally obtain the \emph{reduced $i$-signature of $\lambda$,} which we denote by $w_{e,\s}(\lambda)$.\\
The nodes corresponding to $-_i$ in $w_{e,\s}(\lambda)$ are called $i$-normal.\\
The nodes corresponding to $+_i$ in $w_{e,\s}(\lambda)$ are called $i$-co-normal.\\
The highest $i$-normal node is called $i$-good.\\
The lowest $i$-co-normal node is called $i$-co-good.
A node is called normal (co-normal, good, co-good) if it is $i$-normal ($i$-co-normal, $i$-good, $i$-co-good) for some $i$.

We use these operators to define some directed graph with coloured edges.
\begin{definition}\label{def_crystalfockspacehighestweight}
 Let $\PP_{n, r}$ be the set of all $r$-multipartitions of $n$ and set $\PP_r:=\cup_{n\in \Z_{\geq 0}} \PP_{n,r}$. \\
 On $\PP_r$ define operators $\tei$ and $\tfi$ by 
 \[
  \tei(\lambda)=\begin{cases}
                 \lambda\setminus \{x\}, \quad \text{if } x \text{ is the $i$-good node of $\lambda$}\\
                 0, \quad \text{if $\lambda$ does not have an $i$-good node}
                \end{cases}
 \]
 and 
  \[
  \tfi(\lambda)=\begin{cases}
                 \lambda\cup \{x\}, \quad \text{if } x \text{ is the $i$-co-good node of $\lambda$}\\
                 0, \quad \text{if $\lambda$ does not have an $i$-co-good node}.
                \end{cases}
 \]
 By $B_e(\mathcal{F}(\Lambda_{\s}))$ we denote the directed graph with vertex set $\PP_r$ and edges $\lambda\stackrel{i}{\rightarrow}\mu$ if and only if $\tfi(\lambda)=\mu$, or, equivalently, 
 $\tei(\mu)=\lambda$.\\
 Furthermore, we denote by $B_e(\s)$ the subgraph of $B_e(\mathcal{F}(\Lambda_{\s}))$ defined by the connected component containing the empty partition $\emptyset\vdash_r 0$.
 We call the elements of $B_e(\s)$ \emph{Kleshchev multipartitions}.\\
 Finally, for $\lambda\in \PP_r$ we define 
 \[
  \varphi_i(\lambda):=\max\left\{j\geq 0\mid \tfi^j(\lambda)\neq 0\right\}\quad \text{and}\quad \vareps_i(\lambda):=\max\left\{j\geq 0\mid \tei^j(\lambda)\neq 0\right\}.
 \]
 
 To complete the definition for $r=0$ we define $B_e(\mathbf{0})$ to the be directed graph with exactly one vertex and no edges, where we denote by $\mathbf{0}$  the empty sequence and for consistency we set $\widetilde{\Z_{\geq 0}^0}:=\{\mathbf{0}\}$.
\end{definition}
We can now describe the crystal graphs associated to irreducible highest weight modules.

 \begin{Lemma}
 \label{Lemma_CrystalOfFockSpace}
 Let $\lambda\in P^+$.
 The crystal graph associated to $L(\lambda)$ is isomorphic to $B_e(\s_\lambda)$.
\end{Lemma}
\begin{proof}
By Corollary \ref{Cor_GraphHighestWeightMod}, we can assume without loss of generality that $\lambda\in \overline{P^+}=\bigoplus_{i=0}^{e-1}\Z_{\geq 0}\Lambda_i$.
 Let $t\geq 0$ be defined by $\s_\lambda$ lying in $\widetilde{\Z_{\geq 0}^t}$. \\
 First assume $t>0$. 
 The module $L(\lambda)$ admits a quantum deformation $L_v(\lambda)$ whose specialisation at $v=1$ is isomorphic to $L(\lambda)$. 
 By \cite[Thm 11.8]{ArikiBook}, the crystal graph associated to $L_v(\lambda)$ by a lower crystal basis is isomorphic to $B_e(\s_\lambda)$.
 The discussion of crystals and perfect bases in \cite[Sec 1.5, 1.6]{DudasVaraVass} then shows that this is isomorphic to the crystal graph afforded by a perfect basis of $L(\lambda)$.\\
 Now assume $t=0$. 
 Then $\lambda=0$, because we assumed $\lambda \in \overline{P^+}$.
 By \cite[Rem 2.4.5]{HongKang} and the definition of irreducible highest weight modules we see that $L(0)$ is one-dimensional as a vector space and the $e_i$ and the $f_i$ act as zero on $L(0)$. 
 Thus, any basis of $L(0)$ is a perfect basis and the associated crystal graph is isomorphic to $B_e(\mathbf{0})$.
\end{proof}
The graph $B_e(\mathcal{F}(\Lambda_{\s}))$, too, is a crystal graph.
Namely, its the crystal graph associated to a Fock space, whose definition we give below.
\begin{Definition}[cf. {\cite[Def 10.9]{ArikiBook}}]
 To $\Lambda_{\s}$ we associate an $\sl_e$-module $\mathcal{F}(\Lambda_{\s})$ as follows:
 As a vector space, a basis is given by symbols $|\lambda, \s\rangle$, where $\lambda$ ranges over all $r$-multipartitions of $n$ for all $n\geq 0$.
 Let $\lambda\vdash_r n$.
 For $0\leq i\leq e-1$, we have $e_i|\lambda, \s\rangle=\sum_{\mu}|\mu, \s\rangle$, where $\mu$ ranges over all multipartitions obtained from $\lambda$ by removing an $i$-removable node.
 Similarly, we have $f_i|\lambda, \s\rangle=\sum_{\nu}|\nu, \s\rangle$, where $\nu$ rages over all multipartitions obtained from $\lambda$ by adding an $i$-addable node.
 Moreover, we have $h_i|\lambda, \s\rangle=N_i(\lambda)|\lambda, \s\rangle$, where $N_i(\lambda)$ is the number of $i$-addable nodes of $\lambda$ minus the number of its $i$-removable nodes.
 Finally, we have $d|\lambda, \s\rangle=-M_0(\lambda)$, where $M_0(\lambda)$ is the total number of $0$-nodes of $\lambda$.
\end{Definition}

\begin{Lemma}
 The module $\mathcal{F}(\Lambda_{\s})$ is in $\Oint$ and its associated crystal graph is isomorphic to $B_e(\mathcal{F}(\Lambda_{\s}))$.
\end{Lemma}
\begin{proof}
 This follows from \cite[11.8]{ArikiBook} and \cite[Sec 1.5, 1.6]{DudasVaraVass} as in the proof of Lemma \ref{Lemma_CrystalOfFockSpace}.
\end{proof}
\begin{remark}
 Let $\mathbf{t}=(t_1,\dots, t_r)\in \Z^r$.
 If the multisets $\{t_1\pmod{e},\dots, t_r\pmod{e}\}$ and $\{s_1\pmod{e},\dots, s_r\pmod{e}\}$ are equal, then we have $\Lambda_{\mathbf{t}}=\Lambda_{\s}$ and therefore $\mathcal{F}(\Lambda_{\s})=\mathcal{F}(\Lambda_{\mathbf{t}})$.
\end{remark}

We summarise our description of crystal graphs associated to arbitrary elements of $\Oint$.
 
\begin{Corollary}\label{Cor_CrystalIsomorphismToDirectSum}
 Let $M\in \Oint$ and suppose that $M\cong \oplus_{j} L(\lambda_j)$ for $\lambda_j \in P^+$ is a decomposition into irreducible highest weight modules.
 Then the crystal graph of $M$ is isomorphic to $\coprod_{j} B_e(\s_{\lambda_j})$, where $\coprod$ denotes the disjoint union of graphs.
\end{Corollary}

Following this, we can obtain information about modules in $\Oint$ via solely combinatorial observations, so we study the graphs $B_e(\s)$ and $B_e(\mathcal{F}(\Lambda_{\s}))$ in some more detail.

The first two results are well-known.
\begin{Lemma}\label{Lem_LambdaHasGoodNode}
Let $r\geq 1$ and $\s\in\widetilde{\Z_{\geq 0}^r}$.
 If $\lambda$ is a vertex in $B_e(\s)$, then either $\lambda=\emptyset\vdash_r 0$ or $\lambda$ has an $i$-good node.
\end{Lemma}
\begin{Lemma}\label{Lemma_NUmberOfBoxesIsVarphi}
 Let $r\geq 1$,  $\lambda\vdash_r n$, and $0\leq i\leq e-1$.\\
 The number of $i$-co-normal nodes of $\lambda$ is exactly $\varphi_i(\lambda)$.
 Similarly, the number of $i$-normal nodes of $\lambda$ is exactly $\vareps_i(\lambda)$.
\end{Lemma}

The following is an easy but key observation on co-normal nodes.
\begin{Lemma}\label{Lemma_MoreConormalThanNormal}
 Let $\lambda\vdash_r n$. Then $\lambda$ has exactly $r$ more addable than removable nodes.
 Hence, the number of co-normal nodes of $\lambda$ is exactly $r$ larger than the number of normal nodes.\\
 By Lemma \ref{Lemma_NUmberOfBoxesIsVarphi}, this is equivalent to saying that $\sum_i \left(\varphi_i(\lambda)-\vareps_i(\lambda)\right)=r$.
\end{Lemma}
\begin{proof}
Consider the Young diagram of $\lambda^{(1)}$.
 If $x$ is a removable node of $\lambda^{(1)}$ then there is an addable node in the row directly below $x$ and we can pair of removable and addable nodes in this manner.
 But then we are left with the addable node in the very first row of $\lambda^{(1)}$, which does not have a removable node above it. 
 Hence, $\lambda^{(1)}$ has exactly one more addable than removable node and the same holds true for $\lambda^{(2)}, \dots, \lambda^{(r)}$, so in total $\lambda$ has exactly $r$ more addable than removable nodes.

 Now for co-normal nodes:\\
 Consider all (non-reduced) $i$-signatures over all $0\leq i\leq e-1$  at once. 
 The symbols $-_i$ are in one-to-one correspondence with the removable $i$-nodes and analogously for the symbols $+_i$ and addable $i$-nodes. 
 Thus, the difference between the number of $+$'s and the number of $-$'s, summing over all $i$, is exactly $r$ by what we have just shown.
 In the procedure to compute the reduced $i$-signatures we always remove pairs consisting of one $-$ and one $+$, so this difference remains constant throughout every reduction step.
 Once all $i$-signatures are reduced, this difference is exactly the difference between the number of co-normal and normal boxes of $\lambda$.
\end{proof}

We follow \cite[5.1]{Shan} for the definition of $\sl_e$-categorification in the sense of \cite{Rouquier}.
\begin{definition}
 Set $q:=\exp(2\pi\sqrt{-1}/e)\in \C$.\\
 Let $\CC$ be a $\C$-linear artinian abelian category. 
 For any functor $F:\CC\to \CC$ and any $X\in \End(F)$ we call the generalised eigenspace of $X$ acting on $F$ with eigenvalue $a\in \C$ the \emph{$a$-eigenspace of $X$ in $F$}.\\
 Then an \emph{$\sl_e$-categorification on $\CC$} consists of 
 \begin{enumerate}
  \item an adjoint pair $(U,V)$ of exact functors $\CC\to \CC$,
  \item $X\in \End(U)$ and $T\in \End(U^2)$, and
  \item a decomposition $\CC=\oplus_{\lambda \in P}\CC_{\lambda}$,
  \end{enumerate}
  satisfying the following:
  Set $U_i$ (resp. $V_i$) to be the $q^i$-eigenspace of $X$ in U (resp. in $V$) for $0\leq i\leq e-1$. 
  Then 
  \begin{enumerate}[label=\roman*)]
    \item it is $U=\oplus_{i=0}^{e-1}U_i$,
    \item the endomorphisms $X$ and $T$ satisfy the relations 
    \begin{gather*}
     (1_UT)\circ (T1_U)\circ(1_UT)=(T1_U)\circ(1_UT)\circ(T1_U),\\
     (T+1_{U^2})\circ T-q1_{U^2}=0,\\
     T\circ(1_UX)\circ T=qX1_U,
    \end{gather*}
    \item the map $U_i\mapsto e_i$ and $V_i\mapsto f_i$ for $0\leq i\leq e-1$ defines an integrable representation of $\sl_e$ on the complexification $K_0(\CC):=\C\otimes_{\Z}R_0(\CC)$ of the Grothendieck group,
    \item $U_i(\CC_\lambda)\subseteq \CC_{\lambda+\alpha_i}$ and $V_i(\CC_\lambda)\subseteq \CC_{\lambda-\alpha_i}$, where $\alpha_i$ is the $i$'th simple root of $\sl_e$,
    \item $V$ is isomorphic to a left adjoint of $U$.
 \end{enumerate}
\end{definition}

We fix a $\C$-linear artinian abelian category $\CC$ possessing an $\sl_e$-categorification afforded by an adjoint pair of functors $(U, V)$ and endomorphisms $X$ and $T$.
\begin{Proposition}[ { \cite[Prop 5.20]{ChuangRouquier}, \cite[6.2]{Shan}}]\label{Prop_ChuangRouquier}
Let $0\leq i\leq e-1$. 
Then the data $(U_i, V_i, X, T)$ yields an $\mathfrak{sl}_2$-categorification on $\CC$ in the sense of \cite[Sec 5]{ChuangRouquier}.\\
For an object $M\in \CC$ set $\widetilde{U_i}(M):=\soc(U_i(M))$ and $\widetilde{V_i}(M):=\head(V_i(M))$.
 If $S$ is simple, then $\widetilde{U_i}(S)$ is either simple or $0$. 
 Similarly, $\widetilde{V_i}(S)$ is either $0$ or simple.
 Moreover, if $\widetilde{V_i}(S)\neq 0$, then the multiplicity of $\widetilde{V_i}(S)$ in $V_i(S)$ is exactly 
 $\max_{j}\{j\geq 0\mid \widetilde{V_i}^j(S)\neq 0\}$. 
\end{Proposition}
\begin{Proposition}[ {\cite[Prop 6.2]{Shan}}]\label{Prop_Grothendieckgrouphasperfectbasis}
Let $\irr(\CC)$ be the set of simple objects of $\CC$ up to isomorphism.
  The triple \[(\left\{[S]\mid S\in \irr(\CC)\right\}, \ \{\widetilde{U_i}\mid 0\leq i\leq e-1\}, \ \{\widetilde{V_i}\mid 0\leq i\leq e-1\})\] is a perfect basis of $K_0(\CC)$.\\
\end{Proposition}

These preliminaries allow us to prove our key result on $\sl_e$-categorification.
\begin{Theorem}\label{Thm_ConstituentsCategorification}
 Suppose the $\sl_e$-module $K_0(\CC)$ is in $\Oint$ and the decomposition into irreducible highest weight modules is $\oplus_{j}L(\lambda_j)$ for $\lambda_j\in P^+$. 
 Denote by $B(\CC)$ the crystal graph associated to its perfect basis from Proposition \ref{Prop_Grothendieckgrouphasperfectbasis} and let 
 \[
  \Psi:B(\CC)\to \oplus_j B_e(\s_{\lambda_j})
 \]
 be a crystal graph isomorphism as in Corollary \ref{Cor_CrystalIsomorphismToDirectSum}.\\
 Let $S\in \CC$ be simple and $\s\in \widetilde{\Z_{\geq 0}^r}$ such that $\Psi([S])\in B_e(\s)$.  \\
 Then $V(S)$ has at least $r$ constituents. 
\end{Theorem}
\begin{proof}
 Let $\ell_i(S):=\max_{j}\{j\geq 0\mid \widetilde{V_i}^j(S)\neq 0\}$ for $0\leq i\leq e-1$. 
 Since $V=\oplus_i V_i$, we see that $V(S)$ has at least $\sum_i \ell_i(S)$ constituents by Proposition \ref{Prop_ChuangRouquier}. 
 As $\Psi$ is a crystal isomorphism and $\Psi([S])$ is in $B_e(\s)$, it follows from Lemma \ref{Lemma_MoreConormalThanNormal} that $\sum_i\ell_i(S)=\sum_{i}\varphi_i(\Psi([S]))$ is at least $r$.
\end{proof}
A completely analogous argument can be applied if we have a decomposition into Fock spaces.
\begin{Corollary}\label{Cor_CountingOnFockSpace}
Suppose the $\sl_e$-module $K_0(\CC)$ is in $\Oint$ and it has a decomposition into Fock spaces $K_0(\CC)\cong \oplus_{\mathbf{t}}\mathcal{F}(\Lambda_{\mathbf{t}})$ for $\mathbf{t}\in \cup_{r\geq 1}~\widetilde{\Z_{\geq 0}^r}$. 
 Denote by $B(\CC)$ the crystal graph associated to its perfect basis from Proposition \ref{Prop_Grothendieckgrouphasperfectbasis} and let 
 \[
  \Psi:B(\CC)\to \oplus_{\mathbf{t}} B_e(\mathcal{F}(\Lambda_{\mathbf{t}}))
 \] be a crystal graph isomorphism.\\
 Let $S\in \CC$ be simple, $r\geq 1$ and $\mathbf{t}\in \widetilde{\Z_{\geq 0}^r}$ such that $\Psi([S])\in B_e(\mathcal{F}(\Lambda_{\mathbf{t}}))$.  \\
 Then $V(S)$ has at least $r$ constituents.
\end{Corollary}
\begin{remark}
 Note that by Corollary \ref{Cor_GraphHighestWeightMod} we can slightly weaken the hypothesis of Theorem \ref{Thm_ConstituentsCategorification} and Corollary \ref{Cor_CountingOnFockSpace}:
 Their statements still hold if the isomorphism $K_0(\CC)\cong \oplus_{j}L(\lambda_j)$ for $\lambda_j\in P^+$ (respectively $K_0(\CC)\cong \oplus_{\mathbf{t}}\mathcal{F}(\Lambda_{\mathbf{t}})$) is only an isomorphism of ${\sl_e}^\prime$-modules.
\end{remark}

\subsection{Rational cyclotomic Cherednik algebras}
\label{Section_Cherednik}
Fix integers $r\geq 1 $ and $e\geq 2$ and let $\mathbf{t}=(t_1,\dots, t_r)\in \Z^r$. 
Then for every non-negative integer $n$ we can define a \emph{cyclotomic rational Cherednik algebra (or cyclotomic rational double affine Hecke algebra)} $\HH_{n,\mathbf{t}, e}$:
It is a quotient of the smash product of the complex group algebra $\C[W]$, where $W$ is a complex reflection group of type $G(r,1,n)$, with the tensor algebra of $N\oplus N^*$, where $N$ is the $n$-dimensional vector space on which $W$ acts naturally, cf. \cite[3.1]{Shan} for a precise definition.\\
For every such rational Cherednik algebra there exists a module category $\O_{n,\mathbf{t}, e}$ consisting of all $\HH_{n, \mathbf{t}, e}$-modules that are finitely generated and acted locally nilpotently on by a certain subalgebra of $\HH_{n, \mathbf{t}, e}$.\\
In the following we set $\HH_n:=\HH_{n,\mathbf{t}, e}$ and $\O_n:=\O_{n,  \mathbf{t}, e}$.\\
Bezrukavnikov and Etingof defined parabolic induction and restriction functors $\O_n\to \O_{n+1}$ and $\O_n\to \O_{n-1}$ (cf. \cite{BezruEtingof}) which we will denote by $\ind_n$ and $\res_n$, respectively, in the following.
Set $\O:=\oplus_{n\in \Z_{\geq 0}} \O_n$ and 
\[
 \res:=\oplus_{n\geq 0} \res_n \quad \text{  and }  \ind:=\oplus_{n\geq 0} \ind_n.
\]

 \begin{Proposition}[ {\cite[Cor 4.5]{Shan}}]\label{Prop_ShanCrystalIso}
 Category $\O$ possesses an $\sl_e$-categorification for which the pair of adjoint functors is given by $(\res, \ \ind)$ and the $\sl_e$-module $K_0(\O)$ is in category $\Oint$.
 As an $\sl_e$-module, $K_0(\O)$ is isomorphic to the Fock space $\mathcal{F}(\Lambda_{\mathbf{t}})$.
\end{Proposition}

By Corollary \ref{Cor_CountingOnFockSpace}, this implies the following.
\begin{Theorem}\label{Theo_MainTheoremCherednik}
Let $0\neq M\in \O$. 
Then $\ind(M)$ has at least $r$ constituents. 
In particular, if $r\geq 2$, then $\ind(M)$ is always reducible.
\end{Theorem}

We will later slightly strengthen this result in Remark \ref{Rem_CherednikStrongerVersion}.

\section{Parabolic induction on Ariki-Koike algebras}
\label{chapterariki}
In the following we will use our previous results on crystal graphs to give a new lower bound on the number of constituents of parabolically induced modules of Ariki-Koike algebras. 
After some preliminaries on the object at hand we first reduce the problem to the case of so-called $q$-connected parameters. 
Once this is done, we have to differentiate the cases that the parameter $q$ is either $1$ or not equal to $1$. 
In the second case we can mostly apply our earlier results on $\sl_e$-categorification, whereas for $q=1$ some manual computation yields the corresponding result. 
Note that we treat the case of a generic parameter independently, as we do not have an $\sl_e$-categorification result in this case, but we are still able to use our results on the crystal graphs $B_e(\s)$.\\
We close with the analogous result for the closely related degenerate cyclotomic Hecke algebra.
\subsection{Preliminaries}
Throughout this chapter let $K$ be a field. 
For convenience we assume $K$ to be algebraically closed.

We begin with some preliminaries:\\
Let $n$ and $r$ be positive integers and $q,Q_1,\dots, Q_r$ invertible elements of $K$. 
Then the \emph{Ariki-Koike algebra $\hh_{n,r}(q;Q_1,\dots, Q_r)$} is the unital associative $K$-algebra with generators $T_0,\dots, T_{n-1}$ satisfying relations 
 \begin{align*}
(T_0-Q_1)\cdots(T_0-Q_r)&=0&\\
(T_i-q)(T_i+1)&=0  \quad &\text{ for } 1\leq i \leq n-1\\
T_0T_1T_0T_1&=T_1T_0T_1T_0&\\
 T_iT_{i+1}T_i&=T_{i+1}T_iT_{i+1} &\text{ for } 1\leq i\leq n-2\\
 T_iT_j&=T_jT_i&\text{ for } |i-j|>1.
\end{align*}
This algebras has been defined by Ariki-Koike and, independently, by Brou\'e-Malle, cf. \cite{ArikiKoike, BroueMalle}.
It is obvious that a re-ordering of the $Q_i$ does not change the resulting algebra. 
In the following we fix parameters and set $\hh_n:=\hh_{n,r}(q;Q_1,\dots, Q_r)$.\\
It is well-known that the elements $T_1,\dots T_{n-1}$ generate an Iwahori-Hecke algebra of type $A_{n-1}$ with parameter $q$.
Hence, as usual, for $w\in \sy_n$ we set $T_w:=T_{i_1}\cdots T_{i_k}$ whenever $w=s_{i_1}\cdots s_{i_k}$ is a reduced expression of $w$ in the generators $s_i:=(i,i+1)$.
We define the \emph{Jucys-Murphy elements} $L_i$ inductively by setting $L_1:=T_0$ and $L_{i+1}:=q^{-1}T_iL_iT_i$ for $1\leq i\leq n-1$. 
It has been shown by Ariki and Koike in \cite{ArikiKoike} that 
\[
 \left\{L_1^{a_1}\cdots L_n^{a_n}T_w\mid w\in \sy_n, 0\leq a_1,\dots, a_n\leq r-1\right\}
\]
is a $K$-basis of $\hh_n$.

This implies that $\hh_{n}$ is a subalgebra of $\hh_{n+1}$ and that $\hh_{n+1}$ is free as a left $\hh_n$-module. 
Hence, there exists an exact induction functor
\[
 \ind_n:=\ind_{\hh_n}^{\hh_{n_+1}}:\hh_n\modcat\to \hh_{n+1}\modcat; M\mapsto M\otimes_{\hh_n} \hh_{n+1}.
\]
As this functor is exact it yields a homomorphism of Grothendieck groups 
\[
 R_0(\hh_n)\to R_0(\hh_{n+1}); [M]\mapsto [\ind_n (M)],
\]
where  $[M]$ is the class of the $\hh_n$-module $M$ in the Grothendieck group.
By slight abuse of notation we denote this homomorphism, too, by $\ind_n$.

For every $r$-multipartition $\lambda\vdash_r n$ there exists a well-defined finite dimensional $\hh_n$-module $S^\lambda$ called a \emph{Specht module}, defined in \cite{DJM}.\\
On each Specht module $S^\lambda$ there exists a well-defined bilinear form whose radical $\rad S^\lambda$ is an $\hh_n$-submodule of $S^\lambda$ and we set $D^\lambda:=S^\lambda/\left(\rad S^\lambda\right)$.
These modules fit neatly into the concept of viewing $\hh_n$ as a cellular algebra and in \cite{DJM} it is shown that 
the set 
 \[
  \left\{D^\lambda\mid \lambda\vdash_r n,\ D^\lambda\neq 0\right\}
 \]
is a complete set of pairwise non-isomorphic $\hh_n$-modules.\\
Furthermore, they deduce that the Grothendieck group $R_0(\hh_n)$ is generated by $\left\{[S^\lambda]\mid \lambda\vdash_r n\right\}$.

\subsection{Reduction to \texorpdfstring{$q$}{q}-connected parameter sets}

Many questions on the representation theory of Ariki-Koike algebras have only been covered for so-called $q$-connected parameter sets.
\begin{definition}\label{Def_qconnected}
 Two elements $x$ and $y$ of $K$ are called \emph{$q$-connected} if there exists an integer $k$ such that $x=q^ky$. 
 We write $x\sim_q y$. Clearly, this defines an equivalence relation on $K$.
 We call a set or sequence $X$ with elements in $K$ $q$-connected if all elements of $X$ are $q$-connected.
 Finally, if $X$ and $Y$ are $q$-connected sets (or sequences) over $K$ we say that $X$ and $Y$ are $q$-connected if there exist elements $x\in X$ and $y\in Y$ such that $x$ and $y$ are $q$-connected.
\end{definition}

We set $\QQ:=(Q_1,\dots, Q_r)$.
As reordering of the $Q_i$ does not change the algebra $\hh_n$ we can assume without loss of generality that 
\[
 \QQ=\QQ_1\coprod\dots \coprod \QQ_t, 
\]
for $q$-connected sequences $\QQ_i$ which are pairwise not $q$-connected, where $\coprod$ denotes the concatenation of sequences.
In particular, $t$ is the number of $\sim_q$-equivalence classes on $\QQ$.\\
For $1\leq j\leq t$ define $r_j:=|\QQ_j|$, the length of $\QQ_j$.
Throughout this section we denote by $\otimes$ the tensor product over $K$.\\
\begin{Theorem}[ {\cite[Thm 1.1]{DipperMathas}}]\label{Thm_MoritaEq}
There is a Morita equivalence
\[
 \hh_{n}\thicksim_{\text{Morita}}\hh_{n}^t:=\bigoplus_{\substack{0\leq n_1,\dots,n_t\leq n,\\\sum_i n_i=n}} {^1\hh_{n_1}\otimes \dots \otimes {^t\hh_{n_t}}},
\]
where $^j\hh_{m}:=\hh_{m,\ r_j}(q, \QQ_j)$ for $0\leq m\leq n$ and $1\leq j\leq t$.
In particular, there is an exact functor 
\[
 F_n: \hh_{n}\modcat\to \hh_n^t\modcat.
\]
\end{Theorem}

\begin{remark}
 As the functor $F_n$ is exact it induces a homomorphism on the corresponding Grothendieck groups.
 By abuse of notation we will denote it, too, by $F_n$.\\
 Note that the Grothendieck group of $\hh_n^t$ is the direct sum of the Grothendieck groups of the algebras $^1\hh_{n_+1}\otimes\dots\otimes {^t}\hh_{n_t}$. 
 The irreducible modules of $^1\hh_{n_+1}\otimes\dots\otimes {^t}\hh_{n_t}$ are exactly the tensor products of irreducible modules of $^1\hh_{n_1},  \dots, {^t}\hh_{n_t}$.
\end{remark}

To study the functor $F_n$ we first consider its images on Specht modules. 
\begin{Proposition}[ {\cite[Prop 4.11]{DipperMathas}}]\label{MoritaOnSpecht}
 Let $\lambda=(\lambda^{(1)},\dots,\lambda^{(r)})$ be an $r$-multipartition of $n$. Then define $^1\lambda$ to be the $r_1$-multipartition consisting of the first $r_1$ components of $\lambda$, i.e. 
 $^1\lambda:=(\lambda^{(1)},\dots, \lambda^{(r_1)})$.
 Then let $^2\lambda$ be the $r_2$-multipartition consisting of the next $r_2$ components of $\lambda$, i.e. $
 ^2\lambda=(\lambda^{(r_1+1)}, \dots, \lambda^{(r_1+r_2)})$,etc.
 In the end we have $\lambda={ ^1}\lambda\coprod\dots\coprod {^t}\lambda$. \\
 Then for the Specht module $S^\lambda$ it is
 \[
  F_n(S^\lambda)\cong S^{^1\lambda}\otimes\dots\otimes S^{^t\lambda},
 \]
 which is an $\hh_{n}^s$-module on which nearly all direct summands act as zero with the exception of $^1\hh_{|{^1}\lambda|}\otimes \dots\otimes {^t}\hh_{|{^t}\lambda|}$ .\\
 A completely analogous result holds for the module $D^\lambda$, where we just replace every $S$ by $D$.
\end{Proposition}
The induction $\ind_n$ on Specht modules is well-understood by the following result by Mathas:
\begin{Proposition}[ { \cite[Thm A]{MathasFiltration}}]\label{SpechtFiltlinks}
 Let $\lambda$ be an $r$-multipartition of $n$. 
The induced module $\ind_n(S^\lambda)$ has a filtration $0=I_0\subseteq I_1\subseteq\dots \subseteq I_a=\ind_n(S^\lambda)$ such that for all $1\leq j\leq a$ the quotient $I_j/I_{j-1}$ is also a Specht-module.
Moreover, the Specht modules appearing as such quotients $I_i/I_{i-1}$ are exactly those indexed by the multipartitions of $n+1$ obtained by adding exactly one addable node to $[\lambda]$ and every such multipartition appears exactly once.
\end{Proposition}

We now move towards an analogous result for $\hh_n^t$, at least on the level of Grothendieck groups.
This requires the definition of a number of homomorphisms $R_0(\hh_n^t)\to R_0(\hh_{n+1}^t)$:\\
For $0\leq n_1,\dots, n_t\leq n$ with $\sum_i n_i=n$ and $1\leq j\leq t$ we define 
\[
 ^j\ind_{n,t}^{(n_1,\dots, n_t)}:R_0\left(\hh_n^t\right)\to R_0\left(\hh_{n+1}^{t}\right)
\]
by giving its image on classes of irreducible modules. 
If $M$ is an irreducible module of $\hh_n^t$ that is not in $\left({^1}\hh_{n_1}\otimes\dots\otimes {^t}\hh_{n_t}\right)\modcat$, then set $^j\ind_{n,t}^{(n_1,\dots, n_t)}([M]):=0$.
If $M$ is an irreducible module of $\hh_n^t$ and in $\left({^1}\hh_{n_1}\otimes\dots\otimes {^t}\hh_{n_t}\right)\modcat$, then it is isomorphic to the tensor product 
$D_1\otimes \dots \otimes D_t$ for irreducible $^i\hh_{n_i}$-modules $D_i$. 
In this case we set
\[
^j\ind_{n,t}^{(n_1,\dots, n_t)}([M]):=[D_1\otimes\dots \otimes D_{j-1}\otimes \left(\ind_{^j\hh_{n_j}}^{^j\hh_{n_j+1}}\left(D_j\right)\right)\otimes D_{j+1}\otimes\dots\otimes D_t],
\]
i.e. we apply the usual parabolic induction in the $j$-th component.\\
By Proposition \ref{SpechtFiltlinks}, this immediately yields the following.
\begin{Lemma}\label{Lemma_IndOnOneSpecht}
 For $1\leq i\leq t$ with $i\neq j$ let $M_i\in {^i}\hh_{n_i}\modcat$.
 Let $\alpha\vdash_{r_j} n_j$ be a multipartition. 
 Then it is 
\[
 ^j\ind_{n,t}^{(n_1,\dots, n_t)}([M_1\otimes\dots\otimes S^\alpha\otimes \dots\otimes M_t])= \sum_{\substack{\beta\vdash_{r_j}n_j+1,\\ |\beta\setminus \alpha|=1}} \left[M_1\otimes\dots\otimes S^\beta\otimes\dots\otimes M_t\right],
\]
i.e. all partitions $\beta$ obtained by adding exactly one node to $\alpha$ appear exactly once. 
\end{Lemma}

Now we set \[\ind_{n,t}^{(n_1,\dots, n_t)}:=\sum_{j=1}^t {^j}\ind_{n,t}^{(n_1,\dots, n_t)}\]
and finally\[\ind_{n,t}:=\sum_{\substack{0\leq n_1,\dots,n_t\leq n,\\\sum_i n_i=n}} \ind_{n,t}^{(n_1,\dots, n_t)}.\]
\begin{Theorem}\label{Theorem_MoritaTheoremOnGrothendieck}
 The following diagram commutes:
 \begin{center}
 \begin{tikzcd}
 R_0\left(\hh_n\right)\arrow{r}{F_n}  \arrow{d}{\ind_n} & R_0\left(\hh_{n}^t\right)\arrow{d}{\ind_{n,t}}\\
 R_0\left(\hh_{n+1}\right)\arrow{r}{F_{n+1}}& R_0\left(\hh_{n+1}^t\right)
\end{tikzcd}
\end{center}
\end{Theorem}
\begin{proof}
Let $\lambda\vdash_r n$ and define $^1\lambda, \dots, {^t}\lambda$ as in Proposition \ref{MoritaOnSpecht}. 
For $1\leq i\leq t$ set $n_i:=|{^i}\lambda|$. 
By definition, it is $\ind_{n,t}([F_n(S^\lambda)])=\ind_{n,t}^{(n_1,\dots, n_t)}([F_n(S^\lambda)])$ and, by Lemma \ref{Lemma_IndOnOneSpecht} and the definition of $\ind_{n,t}^{(n_1,\dots, n_t)}$, we have 
\[
 \ind_{n,t}^{(n_1,\dots, n_t)}([F_n(S^\lambda)])=\sum_{j=1}^t \sum_{\substack{\beta\vdash_{r_j} n_j+1\\ |\beta\setminus {^j}\lambda|=1}}
 [S^{^1\lambda}\otimes\dots\otimes S^{^{j-1}\lambda}\otimes S^{\beta}\otimes S^{^{j+1}\lambda}\otimes \dots\otimes S^{^t\lambda}].
\]
For $1\leq j\leq t$ and $\beta\vdash_{r_j} n_j$ with $|\beta\setminus {^j}\lambda|=1$ let $\mu(j, \beta)\vdash_r n$ be the multipartition of $n+1$ obtained as the concatenation $({^1}\lambda, \dots, \beta, \dots, {^t}\lambda)$, where $\beta$ is the $j$'th subsequence. 
By Proposition \ref{MoritaOnSpecht} it is $F_{n+1}([S^{\mu(j, \beta)}])=[S^{^1\lambda}\otimes\dots\otimes S^{^{j-1}\lambda}\otimes S^{\beta}\otimes S^{^{j+1}\lambda}\otimes \dots\otimes S^{^t\lambda}]$.\\
Clearly, the $\mu(j, \beta)$ run exactly over all multipartitions of $n+1$ which are obtained from $\lambda$ by adding exactly one node and every such multipartition appears exactly once. 
Hence, by Proposition \ref{SpechtFiltlinks} it is 
\begin{align*}
 \ind_n([S^\lambda])\ =&\ \sum_{j=1}^t \sum_{\substack{\beta\vdash_{r_j} n_j+1}} [S^{\mu(j, \beta)}]
\end{align*}
 and thus $F_{n+1}(\ind_n([S^\lambda]))= \ind_n^t(F_n([S^\lambda]))$.\\
Since the classes of Specht modules generate the Grothendieck group of $\hh_n$ this already implies the commutativity of the diagram.
\end{proof}

\begin{Corollary}\label{MoritaNumbers}
 The homomorphisms $F_n$ and $F_{n+1}$ obtained from the Morita equivalence preserve the number of irreducible constituents.
 Hence, if $\lambda=(\lambda^{(1)},\dots, \lambda^{(r)})$ is an $r$-multipartition of $n$ such that $D^\lambda\neq 0$, then the number of irreducible constituents of the module $\ind_n(D^\lambda)$ is equal to that of $\ind_{n,t}\left(D^{^1\lambda}\otimes\dots\otimes D^{^t\lambda}\right)$ by Proposition \ref{MoritaOnSpecht}.
By definition, this is equal to the number obtained by summing the number of constituents of $\ind_{^j\hh_{n_j}}^{^j\hh_{n_j+1}}(D^{^j\lambda})$ over all $j$, where $n_j:=|{^j}\lambda|$.\\
Since $^j\hh_{n_j}$ and $^j\hh_{n_j+1}$ are defined over $q$-connected parameters $\QQ_j$ we have now reduced the problem of finding the number of constituents of induced modules to the $q$-connected parameter case.
\end{Corollary}
\begin{remark}
 A completely analogous result holds for restriction in place of induction. The filtration in Proposition \ref{SpechtFiltlinks} has to be replaced by that in \cite{MathasFiltrationRestriction} for restriction of Specht modules.
 One has to pay attention when defining the partial restriction homomorphisms $ {^j\res_{n,t}^{(n_1,\dots,n_t)}}$; they are only defined for $n_j\geq 1$.
 Then they are defined on classes of irreducibles via the usual parabolic restriction in the $j$'th component of the tensor product. \\
 In total, we define the restriction on the $\hh_n^t$ side to be 
 \[
  \res_{n,t}=\bigoplus_{j=1}^t\bigoplus_{\substack{0\leq n_1,\dots,n_t\leq n,\\\sum_i n_i=n\\n_j\geq 1}} {^j\res_{n,t}^{(n_1,\dots,n_t)}}.
 \]
 Then the following diagram commutes:
 \begin{center}
  \begin{tikzcd}
 R_0\left(\hh_n\right)\arrow{r}{F_n}  \arrow{d}{\res_n} & R_0\left(\hh_{n}^t\right)\arrow{d}{\res_{n,t}}\\
 R_0\left(\hh_{n-1}\right)\arrow{r}{F_{n-1}}& R_0\left(\hh_{n-1}^t\right)
\end{tikzcd}
\end{center}
\end{remark}
\subsection{\texorpdfstring{$q$}{q}-connected parameters}
Assume that $\QQ=(Q_1,\dots, Q_r)$ is $q$-connected, as we have just reduced our problem to this case.\\
We can further simplify the setting without loss of generality:
Let $a\in K^*$. Then $\hh_{n,r}(q;Q_1,\dots, Q_r)$ is isomorphic to $\hh_{n,r}(q; aQ_1,\dots, aQ_r)$ by replacing $T_0$ with $a^{-1} T_0$. 
Hence, if $\QQ$ is $q$-connected we can assume without loss of generality that there exist non-negative integers $s_1,\dots, s_r$ such that $Q_i=q^{s_i}$ for $1\leq i\leq r$, and will assume this to be the case from now on.
Now let $e\in \Z_{\geq 0}\cup \{\infty\}$ be the multiplicative order of $q$ in $K^*$. 
Then we can additionally assume that $0\leq s_1,\dots, s_r<e$ and as reordering of the $Q_i$ does not change $\hh_n$ we also assume $s_1\leq s_2\leq\dots\leq s_r$.
Set $\s:=(s_1,\dots, s_r)\in \widetilde{\Z_{\geq 0}}$.

In the following we will have to differentiate the cases $q\neq 1$ and $q= 1$.\\
Note that in the latter case people will often switch to considering so-called degenerate cyclotomic Hecke algebras instead, which are slightly different than the Ariki-Koike algebras for $q=1$ we consider here, but the definitions of Ariki-Koike algebras make sense, too, for $q=1$, so we see no reason to exclude this case. 
However, for completeness we will remark in Theorem \ref{Thm_MainTheoremDegenerate} how to handle degenerate cyclotomic Hecke algebras.

\subsubsection{The case \texorpdfstring{$q\neq 1$}{q not 1}}
Set $\hh\modcat:=\bigoplus_{n\geq 0}\hh_n\modcat$ and 
 \[
  \ind:=\bigoplus_{n\geq 0} \ind_n\quad \res:=\bigoplus_{n\geq 0}\res_n.
 \]
We first consider the case $e<\infty$. 
This is where our previous results do come in:
\begin{Proposition}[ { \cite[Exp 5.2.5]{ShanPHD}, \cite[Thm 12.5]{ArikiBook}, \cite[Thm 6.1]{ArikiBranching} }]\label{Prop_AKKCategorification}
The functors $\res$ and $\ind$ constitute a pair of bi-adjoint functors on $\hh\modcat$, yielding an $\sl_e$-categorification.
The $\sl_e$-module $K_0(\hh\modcat)$ is isomorphic to the irreducible highest weight module $L(\Lambda)$ with $\Lambda=\Lambda_{s_1}+\dots+\Lambda_{s_r}$.
If we denote by $B_e(\hh)$ the crystal graph associated to $K_0(\hh\modcat)$, then $B_e(\hh)$ is isomorphic to $B_e(\s)$
and the pre-image of the vertex $\emptyset\vdash_r 0$ is the class of the trivial module of the trivial algebra $\hh_0$.
\end{Proposition}

\begin{Proposition}\label{Prop_MainThme<infty}
Suppose $2\leq e<\infty$.
 Let $n\geq 1$ and $0\neq M\in \hh_n\modcat$. Then $\ind_n(M)$ has at least $r+1$ irreducible constituents. 
 In particular, $\ind_n(M)$ is reducible.
\end{Proposition}
\begin{proof}
Since $\ind$ is exact it suffices to consider the case that $M$ is irreducible.
Let $\Psi:B_e(\hh)\to B_e(\s)$ be the crystal graph isomorphism from Proposition \ref{Prop_AKKCategorification}.
By Proposition \ref{Prop_ChuangRouquier} and the fact that $\Psi$ is a crystal isomorphism, we see that the number of constituents of $\ind_n(M)$ is at least $\sum_{i=0}^{e-1}\varphi_i(\Psi([M]))$. 
Since $\Psi([M])$ is not the empty partition(and so has at least one removable node), by Lemmas \ref{Lem_LambdaHasGoodNode}, \ref{Lemma_NUmberOfBoxesIsVarphi}, and \ref{Lemma_MoreConormalThanNormal}, this number is at least $r+1$.
\end{proof}

\begin{Remark}\label{Rem_CherednikStrongerVersion}
We can use Proposition \ref{Prop_MainThme<infty} to strengthen the statement of Theorem \ref{Theo_MainTheoremCherednik} on certain modules.
Assume $K=\C$ and suppose $q=\exp(2\pi\sqrt{-1}/e)$.
Let $\HH_n:=\HH_{n, \mathbf{s}, e}$ be a cyclotomic rational Cherednik algebra and $\O_n:=\O_{n, \mathbf{t}, e}$ the corresponding module category, cf. \ref{Section_Cherednik}.
Then there exists a well-defined exact functor  $\KZ_n:\O_n\to \hh_n\modcat$.
For any irreducible $M$ in $\O_n$ the $\hh_n$-module $\KZ_n(M)$ is either irreducible or $0$, hence the number of constituents of $\KZ_n(M)$ is a lower bound for the number of constituents of $M$. 
Thus, by Proposition \ref{Prop_MainThme<infty}, if $n\geq 1$ and $M\in\O_n$ such that $\KZ_n(M)\neq 0$, then $\ind_{\HH_n}^{\HH_{n+1}}(M)$ has at least $r+1$ constituents
\end{Remark}

Now suppose $e=\infty$. \\
This case has been studied in detail by Vazirani in \cite{VazII}. 
Here, too, we obtain a result using crystal graphs. 
The crystal graph $B_\infty(\s)$ is defined just as for $e<\infty$, if we define $x\equiv y\pmod{\infty}$ if and only if $x=y$ for integers $x$ and $y$ to extend the definition of the residue to the case $e=\infty$.
Note that $B_{\infty}$ is a crystal graph for $U_u(\sl_{\infty})$.
\\
For $i\geq 0$ there exists a refined functor $\iind_n\hh_n\modcat\to \hh_{n+1}\modcat$, defined via taking generalised eigenspaces of Jucys-Murphy elements, cf. e.g. \cite{ArikiBranching} for a definition. 
We define $\iind:=\oplus_{n\geq 0} \iind_n$.
\begin{Proposition}[ {\cite{Gro, VazII}}]\label{Prop_PropertiesOfiind}
 The functors $\iind$ satisfy the following:
 \begin{enumerate}
  \item $\iind$ is exact.
  \item For $M\in \hh_n\modcat$ it is $\ind(M)\cong \oplus_{i\geq0}\iind(M)$.
  \item Let $M\in\irr(\hh)$. Then $\tfi(M):=\text{head}(\iind(M))$ is either $0$ or irreducible.
  \item If $M\in\irr(\hh)$ and $\tfi(M)\neq 0$, then the multiplicity of $\tfi(M)$ in $\iind(M)$ is exactly $\max_j\{j\geq 0\mid \tfi(M)^j\neq 0\}$.
  \item Define a directed graph $B_{\infty}(\hh)$ with vertex set $\irr(\hh)$ and directed edges $M\stackrel{i}{\to}N$ for $M\in \irr(\hh_n)$ and $N\in \irr(\hh_{n+1})$ if and only if $\tfi(M)=N$. 
  Then $B_{\infty}(\hh)$ is isomorphic to $B_{\infty}(\s)$ and the pre-image of the vertex $\emptyset\vdash_r 0$ is the class of the trivial module of the trivial algebra $\hh_0$.
 \end{enumerate}
\end{Proposition}

With this in mind the proof of the following is completely analogous to the case $e<\infty$, as Lemmas \ref{Lem_LambdaHasGoodNode}, \ref{Lemma_NUmberOfBoxesIsVarphi}, and \ref{Lemma_MoreConormalThanNormal} all also hold for $e=\infty$.
\begin{Proposition}\label{Prop_MainThme=infty}
Suppose $e=\infty$.
 Let $n\geq 1$ and $0\neq M\in \hh_n\modcat$. Then $\ind_n(M)$ has at least $r+1$ irreducible constituents. 
 In particular, $\ind_n(M)$ is reducible.
\end{Proposition}

\subsubsection{The case \texorpdfstring{$q=1$}{q = 1}}
Now assume $q=1$ and $n\geq 1$.\\
Note that the $q$-connectedness of the $Q_i$ then implies that they are all equal to $1$.
\begin{remark}
 Beware that in general $\hh_n$ is \emph{not} isomorphic to the so-called degenerate cyclotomic Hecke algebra.\\
 Furthermore, for $r>1$ it is in general \emph{not} isomorphic to the group algebra $K[G(r,1,n)]$. 
 For example, if $K=\C$ and $r>1$ we see that $\hh_n$ is not semisimple by \cite[Cor 3.3]{MathasAK} and thus it is clearly not isomorphic to the semisimple algebra $\C[G(r,1,n)]$.
\end{remark}
 
For $q=1$, the subalgebra of $\hh_n$ that is generated by $T_1,\dots, T_{n-1}$ is isomorphic to the group algebra $K[\sy_n]$ and we identify the two.
 The Specht and irreducible modules of $\hh_n$ for $q=1$ have been studied by Mathas.
 Their structure is not overly complicated.
\begin{Proposition}[  {\cite[Theorem 3.7, Lemma 3.3]{Mathasq1}}]
 Let $\lambda=\left(\lambda^{(1)},\dots, \lambda^{(r)}\right)\vdash_r n$ be an $r$-multipartition such that $D^\lambda\neq 0$. 
 Then the following holds:
 \begin{enumerate}
  \item It is $\lambda^{(j)}=\emptyset$ unless $j=r$.
  \item The Jucys-Murphy elements $L_1,\dots, L_n$ act trivially on $S^\lambda$ and hence also on $D^\lambda$.
 \end{enumerate}
\end{Proposition}

Hence, the action on Specht and irreducible modules is completely determined by the restriction to the group algebra $K[\sy_n]$.
In particular, the irreducible $\hh_n$-modules are exactly the irreducible $K[\sy_n]$-modules seen as $\hh_n$-modules by letting $T_0=L_1$ act as $1$.\\
For a partition $\alpha\vdash n$ denote by $S^\alpha$ the Specht module of $K[\sy_n]$ and as usual by $D^\alpha$ its quotient by the radical of the corresponding bilinear form. 
It is well-known that $D^\alpha\neq 0$ if and only if $\alpha$ is $p$-restricted, where $p$ is the characteristic of $K$.

Denote by $\res_{\sy_n}$ the restriction functor $\res_{K[\sy_n]}^{\hh_n}$. Then the next result follows from a close study of the explicit construction of Specht modules for $\hh_n$ and $K[\sy_n]$.
\begin{Lemma}
 Let $\lambda=(\emptyset,\dots, \emptyset, \lambda^{(r)})\vdash_r n$. Then the following holds:
 \begin{enumerate}
  \item The restriction $\res_{\sy_n}(S^\lambda)$ is isomorphic to $S^{\lambda^{(r)}} \in K[\sy_n]\modcat$.\label{firstpart}
  \item The restriction $\res_{\sy_n}(D^\lambda)$ is isomorphic to $D^{\lambda^{(r)}}\in K[\sy_n]\modcat$.\label{secondpart}
 \end{enumerate}
\end{Lemma}
\begin{proof}
 We refer the reader to \cite[Ch 3]{MathasAK} and \cite[Ch 3]{MathasBook} for details on the construction of Specht modules and follow these references.
 In particular, we do not vigorously define everything in this proof but rather assume familiarity with the construction and the necessary vocabulary.\\
 Let us recall, though, that a tableau of shape $\mu\vdash_r n$ is a one-to-one labeling of the nodes of $\mu$ with the numbers $\{1,\dots, n\}$ and that a standard tableau is a special tableau.
 Accordingly, one can define (standard) tableaux of shape $\beta\vdash n$.
 \\ 
 Set $\alpha:=\lambda^{(r)}$. 
 To shorten notation throughout this proof let $\hh:=\hh_n$ and $\hhh:=K[\sy_n]$.\\
 Then $\hh$ has a $K$-basis $M:=\{m_{\mathfrak{u}\mathfrak{v}}\mid \mu\vdash_r n, \ \mathfrak{u},\mathfrak{v} \text{ standard tableaux of shape } \mu\}$ 
 and $\hhh$ has a $K$-basis $M':=\{m_{\mathfrak{a}\mathfrak{b}}'\mid \beta\vdash n, \ \mathfrak{a},\mathfrak{b} \text{ standard tableaux of shape } \beta\}$.
 The key observation is that via identifying $\beta\vdash n$ with $(\emptyset, \dots, \emptyset, \beta)\vdash_r n$ the elements $m_{\mathfrak{a}\mathfrak{b}}'$ and $m_{\mathfrak{a}\mathfrak{b}}$ are equal for all standard tableaux $\mathfrak{a}$ and $\mathfrak{b}$ of shape $\beta$, hence $M'$ embeds into $M$.\\
 Now let $\hh^\lambda$ be the $K$-span of all $m_{\mathfrak{u}\mathfrak{v}}$ where the shape of $\mathfrak{u}$ strictly dominates $\lambda$ and let $\hhh^\alpha$ be the $K$-span of all $m_{\mathfrak{a}\mathfrak{b}}'$ where the shape of $\mathfrak{a}$ strictly dominates $\alpha$. 
 Note that in the first case we consider the dominance order of $r$-multipartitions, whereas in the latter the dominance order of ordinary partitions is used.
 Then $\hh^\lambda\trianglelefteq \hh$ and $\hhh^\alpha\trianglelefteq \hhh$ are two-sided ideals.
 Furthermore, it follows from the definitions that $\hh^\lambda\cap \hhh=\hhh^\alpha$.
 \\
 Now $S^\lambda$ is a submodule of $\hh/\hh^\lambda$ with $K$-basis $\{m_{\mathfrak{u}}+\hh^\lambda\mid\mathfrak{u}\text{ a standard tableau of shape }\lambda \}$, where we set $m_{\mathfrak{u}}:=m_{\mathfrak{t}^\lambda\mathfrak{u}}$ for $\mathfrak{t}^\lambda$ the tableau obtained by labeling $\lambda$ left to right, top to bottom.
 Similarly, via the above identification of tableau of shape $\alpha$ with those of shape $\lambda$ we know that $S^\alpha$ is the submodule of $\hhh/\hhh^\alpha$ with $K$-basis $\{m_{\mathfrak{u}}+\hhh^\alpha \mid \mathfrak{u}\text{ a standard tableau of shape }\lambda\}$.\\
 Since the representatives of the basis elements all lie in $\hhh$ and because $\hh^\lambda\cap \hhh=\hhh^\alpha$, we see that the $K$-vector space isomorphism $\Psi:S^\lambda\to S^\alpha\, ; \, m_{\mathfrak{u}}+\hh^\lambda\mapsto m_{\mathfrak{u}}+\hhh^\alpha$ is an $\hhh$-module isomorphism.
 This proves \ref{firstpart}.\\
 The irreducible modules $D^\lambda$ and $D^\alpha$ are defined as quotients of $S^\lambda$ and $S^\alpha$ by the radical of bilinear forms $\langle\ ,\ \rangle_\lambda$ and $\langle\ ,\ \rangle_\alpha$ on $S^\lambda$ and $S^\alpha$, respectively.
 These forms are defined via a number of of equations in $\hh/\hh^\lambda$ and $\hhh/\hhh^\alpha$, respectively, and from $\hh^\lambda\cap \hhh=\hhh^\alpha$ we can deduce that $\Psi$ respects the forms, i.e. $\langle x, y \rangle_\lambda=\langle \Psi(x),\Psi(y)\rangle_\alpha$ for all $x,y$ in $S^\lambda$.
 Thus, $\Psi$ induces an $\hhh$-isomorphism $D^\lambda\to D^\alpha$, proving \ref{secondpart}.
\end{proof}
\begin{remark}
 The above proof does not require $q$ to be $1$.
 As everything in \cite{MathasBook} is actually carried out for arbitrary Iwahori-Hecke algebras of type $A$, our proof still holds for arbitrary $q$, in which case the subalgebra of $\hh_n$ generated by $T_1,\dots, T_{n-1}$ is an Iwahori-Hecke algebra of type $A$ with parameter $q$.\\
 Finally, note that the Specht modules defined in \cite{MathasBook} are what other authors might call \emph{dual Specht modules} instead.
\end{remark}

Let us study the module structure of the induced module $\ind_n(D^\lambda)$ more carefully.
We begin with a number of technical results.
\begin{Lemma}\label{Lemma_ActionOfJucys}
 Let $0\leq i\leq n+1$. Then the following holds.
 \begin{enumerate}
  \item $(L_i-1)^r=0$. 
  \item For $w\in \sy_n$, we have $wL_i=L_{(i)w^{-1}}w$, where $(i)w^{-1}$ is the image of $i$ under the natural right action of $w\in\sy_n$ on $\{1,\dots, n\}$. 
  In particular we have $wL_{i+1}=L_{i+1}w$.
 \end{enumerate}
\end{Lemma}
\begin{proof}
 For $q=1$ we see that the generator $T_j$ of $\hh_n$ is an involution for $1\leq j\leq n-1$, i.e. $T_j^2=1$. 
 The definition of the $L_i$ then implies that they are all conjugated in $\hh_{n+1}$. 
 As $L_1=T_0$, and $(T_0-1)^r=0$ is one of the defining relations of $\hh_{n+1}$, we see that $(L_i-1)^r=0$ for all $i$.\\
 Similarly, from the definition of the $L_{i}$ one can deduce that 
 $ T_jL_i=\begin{cases}L_{i+1}\,{T\!}_j, \quad &\text{ if } i=j,\\ L_{i-1}\,{T\!}_j,\quad &\text{ if } i-1=j\\ L_i\,{T\!}_j, \quad &\text{ otherwise}\end{cases}$
 for $1\leq j\leq n-1$.
 It follows by induction that $wL_i=L_{(i)w^{-1}}w$.
\end{proof}

\begin{Lemma}\label{Lemma_GoodBasisOfIndDlambda}
 Let $D^\lambda$ be an irreducible $\hh_n$-module and $B$ a $K$-basis of $D^\lambda$. 
 Moreover, let $Y$ be the set of distinguished right coset representatives of $\sy_n$ in $\sy_{n+1}$.
 Then 
 \[
  \{b\otimes_{\hh_n}(L_{n+1}-1)^jy\mid b\in B, \ y\in Y, \ 0\leq j<r\}
 \]
 is a $K$-basis of $\ind_n(D^\lambda)$.
\end{Lemma}
\begin{proof}
 It is well-known that $\hh_{n+1}$ is free over $\hh_n$ as a left module and that a basis is given by $\{L_{n+1}^jy\mid y\in Y, \ 0\leq j<r\}$. 
 It follows, that
  $\{b\otimes_{\hh_n}L_{n+1}^jy\mid b\in B, \ y\in Y, \ 0\leq j<r\}$ is a $K$-basis of $\ind_n(D^\lambda)=D^\lambda\otimes_{\hh_n} \hh_{n+1}$.
 As we have \[b\otimes (L_{n+1}-1)^jy = b\otimes L_{n+1}^jy + \sum_{k=0}^{j-1} (-1)^{j-k}\begin{pmatrix}j\\ k\end{pmatrix}\left(b\otimes L_{n+1}^{k}y\right)\] for $b\in B$, $y\in Y$, $0\leq j<r$, induction on $j$ now shows that $\{b\otimes_{\hh_n}(L_{n+1}-1)^jy\mid b\in B, \ y\in Y, \ 0\leq j<r\}$, too, is a $K$-basis of $\ind_n(D^\lambda)$.
\end{proof}
\begin{Proposition}\label{Prop_ModulestructureIndModdegenerate}
 Assume the setting of Lemma \ref{Lemma_GoodBasisOfIndDlambda}.
 For $0\leq {\ell}< r$ let $M_{\ell}$ be the $K$-vector space spanned by 
 \[
       \left\{b\otimes_{\hh_n} \left(L_{n+1}-1\right)^{j}y\mid b\in B,\ y\in Y,\ \ell\leq j< r \right\}                                                        
 \]
 and set $M_r:=0$.
Then the following holds:
\begin{enumerate}
 \item For every $\ell$, $M_\ell$ is an $\hh_{n+1}$-module.
 \item It is $0=M_r\lneq M_{r-1}\dots\lneq M_{1}\lneq M_0=\ind_n(D^\lambda)$.
 \item Set $N_\ell:=M_{\ell}/M_{\ell+1} $ for $0\leq \ell\leq r-1$.
      Then $T_0=L_1$ acts trivially on $N_\ell$ and $\res_{\sy_{n+1}}(N_\ell)$ is isomorphic to $\widehat{\ind_n}\left(D^{\lambda^{(r)}}\right)$, where we set $\widehat{\ind_n}:= \ind_{K[\sy_n]}^{K[\sy_{n+1}]}$.
      Note that $\widehat{\ind_n}\left(D^{\lambda^{(r)}}\right)$ does not depend on $\ell$.
\end{enumerate}
\end{Proposition}
\begin{proof}
Let $0\leq \ell<r$.
 First note that the spanning set of $M_\ell$ given above is a $K$-basis by Lemma \ref{Lemma_GoodBasisOfIndDlambda}. 
 This already shows $M_{\ell+1}\subsetneq M_\ell$.\\
 To show that $M_\ell$ is a $\hh_{n+1}$-module, it suffices to show that it is closed under the right action of $\sy_{n+1}$ and $L_1$, as these together generate $\hh_{n+1}$. 
 To this end, let $b\in B$, $y\in Y$, $0\leq j<r$ and $w\in \sy_{n+1}$. 
 As $Y$ is a set of right coset representatives of $\sy_n$ in $\sy_{n+1}$ there exist $z\in Y$ and $v\in \sy_n$ such that $yw=vz$. 
 By Lemma \ref{Lemma_ActionOfJucys}, we have $L_{n+1}v=vL_{n+1}$, and thus 
  \begin{align*}
  \left( b\otimes_{\hh_n}(L_{n+1}-1)^jy\right)w= bv\otimes (L_{n+1}-1)^j z,
  \end{align*}
  and the right-hand side is in $M_\ell$ by definition, as $B$ is a $K$-basis of $D^\lambda$.\\
  On the other hand, $L_i$ acts trivially on $D^\lambda$ for every $1\leq i\leq n$, in particular we have $bL_i=b$. 
  With this in mind and Lemma \ref{Lemma_ActionOfJucys} one can show that
  \begin{align*}
   \left( b\otimes_{\hh_n}(L_{n+1}-1)^jy\right)L_1=
   \begin{cases}
   b\otimes_{\hh_n}(L_{n+1}-1)^jy, \quad &\text{if } (1)y^{-1}\neq n+1\\
   b\otimes_{\hh_n}(L_{n+1}-1)^jy+b\otimes_{\hh_n}(L_{n+1}-1)^{j+1}y, \quad &\text{if }(1)y^{-1} = n+1.
   \end{cases}
  \end{align*}
  Again, the right-hand side is in $M_\ell$ by definition and because $(L_{n+1}-1)^r=0$.
  Thus, in total $M_\ell$ is an $\hh_{n+1}$-module and we have $M_{\ell+1}\lneq M_\ell$.\\
  
  Clearly, 
  \[
   \left\{b\otimes_{\hh_n}(L_{n+1}-1)^{\ell}y \ + \  M_{\ell+1}\mid b\in B, \ y\in Y\right\}
  \]
  is a $K$-basis of $N_\ell$.
  By our study of the action $L_1$ on the basis elements we easily see that $L_1$ acts trivially on $N_\ell$.
  Furthermore, we have shown that $w\in \sy_{n+1}$ acts via 
  \[
  \left( b\otimes_{\hh_n}(L_{n+1}-1)^\ell y + M_{\ell+1} \right)w= bv\otimes (L_{n+1}-1)^\ell z +M_{\ell+1}.
  \]
  Now, as $B$ is a $K$-basis of $D^\lambda$, it is also a $K$-basis of the $K[\sy_{n}]$-module $D^{\lambda^{(r)}}\cong\res_{K[\sy_{n+1}]}^{\hh_{n+1}}(D^\lambda)$.
  Since $Y$ is the set of distinguished right coset representatives of $\sy_n$ in $\sy_{n+1}$ we see that the set $\{b\otimes_{K[\sy_n]} y\mid b\in B, \ y\in Y\}$ is a $K$-basis of the induced module $\widehat{\ind_n}\left(D^{\lambda^{(r)}}\right)$.
 The action of $w$ on these basis elements is given by
  \[
   \left(b\otimes_{K[\sy_n]} y\right)w=bv\otimes_{K[\sy_n]} z.
  \]
  In total, it follows that 
  \[
  \widehat{\ind_n}\left(D^{\lambda^{(r)}}\right)\to \res_{K[\sy_{n+1}]}^{\hh_{n+1}}\left(N_\ell\right);\ b\otimes_{K[\sy_n]} y\mapsto b\otimes (L_{n+1}-1)^\ell y + M_{\ell+1}
  \]
  is an isomorphism of $K[\sy_{n+1}]$-modules.
\end{proof}

\begin{Corollary}\label{Cor_NumberOfConstsIsaProduct}
 Let $D^\lambda$ be an irreducible $\hh_n$-module for a multipartition $\lambda=(\emptyset,\dots, \emptyset, \lambda^{(r)})\vdash_r n$.
 Suppose $t\in \N$ is the number of irreducible constituents of the $K[\sy_{n+1}]$-module $\widehat{\ind_n}(D^{\lambda^{(r)}})$. 
 Then the number of irreducible constituents of $\ind_n(D^\lambda)$ is exactly the product $rt$.
\end{Corollary}
\begin{proof}
 We use the notation of Proposition \ref{Prop_ModulestructureIndModdegenerate}.
 The modules $N_\ell$ for $0\leq \ell\leq r-1$ are all isomorphic, as $L_1$ acts trivially on every $N_\ell$, their restrictions $\res_{K[\sy_{n+1}]}^{\hh_{n+1}}\left(N_\ell\right)$ are all isomorphic to $\widehat{\ind_n}(D^{\lambda^{(r)}})$, and $L_1$ and $\sy_{n+1}$ together generate $\hh_{n+1}$.
 In particular, the number of irreducible constituents of every $N_\ell$ is exactly $t$.
  The claim now follows from the fact that the $N_\ell$ are exactly the quotients in the submodule chain 
  \[
   0=M_r\leq \dots\leq M_0=\ind_n(D^\lambda).
  \]
\end{proof}

\begin{Proposition}\label{Prop_MainThme=1}
Suppose $q=1$.
 Let $0\neq M$ be an $\hh_n$-module.
 Then $\ind_n(M)$ has at least $2r$ irreducible constituents.
\end{Proposition}
\begin{proof}
 As induction is exact it suffices to prove the statement for $M=D^\lambda\neq 0$ for a multipartition $\lambda=(\emptyset,\dots, \emptyset, \lambda^{(r)})\vdash_r n$.\\
 The induced module $\widehat{\ind_n}(D^{\lambda^{(r)}})$ has at least $2$ irreducible constituents by \cite[Theorem 1.1]{Schoennenbeck}.
 The claim now follows from Corollary \ref{Cor_NumberOfConstsIsaProduct}.
\end{proof}

We finish this subsection by describing the socle of the induced modules, as this can be obtained with barely any additional work and complements the branching rules for $q\neq 1$, cf. \cite{ArikiBranching, VazII}.
\begin{Proposition}
Assume the setting and notation of Proposition \ref{Prop_ModulestructureIndModdegenerate}. 
Then the socle of $\ind_n(D^\lambda)$ is contained in $M_{r-1}$. 
More precisely, the socle of $\ind_n(D^\lambda)$ is isomorphic to the socle of $\widehat{\ind_n}(D^{\lambda^{(r)}})$ where $T_0$ acts trivially.
\end{Proposition}
\begin{proof}
Clearly, the $L_i$ act trivially on the socle of an $\hh_{n+1}$-module.\\
It is easily checked that the common eigenspace of the $L_i$ with respect to the eigenvalue $1$ on $\ind_n(D^\lambda)$ is exactly $M_{r-1}$.
By Proposition \ref{Prop_ModulestructureIndModdegenerate}, the restriction $\res_{\sy_{n+1}}(M_{r-1})$ is isomorphic to $\widehat{\ind_n}(D^{\lambda^{(r)}})$, yielding the claim.

\end{proof}
\begin{remark} 
The socle of $\widehat{\ind_n}(D^{\lambda^{(r)}})$ has been studied extensively by Kleshchev in his groundbreaking series of papers in the early 90's, cf. \cite{KleshchevBook} for a survey. 
 In particular, he defines refined induction and restriction functors and shows that these can be defined in terms of adding and removing certain nodes.
As this, too, yields a crystal, analogous to the ones defined for $2\leq e\leq \infty$, we could also have used Kleshchev's results instead of \cite[Theorem 1.1]{Schoennenbeck} to show that $\widehat{\ind_n}(D^{\lambda^{(r)}})$ has at least $2$ constituents.
\end{remark}

\subsubsection{Main Theorem}
We drop our conditions to obtain a result on arbitrary Ariki-Koike algebras with invertible parameters:
\begin{Theorem}\label{Main_Theorem_Hecke}
Let $t$ be the number of $\sim_q$-equivalence classes on $\left(Q_1,\dots, Q_r\right)$.
Then for any $\hh_n$-module $M\neq 0$ the number of constituents of $\ind_n(M)$ is at least $r+t$.
In particular, $\ind_n(M)$ is reducible.
\end{Theorem}
\begin{proof}
 Reorder the $Q_i$ such that $\QQ:=(Q_1,\dots, Q_r)=\QQ_1\coprod\dots\coprod \QQ_t$ is a concatenation of $q$-connected sequences that are pairwise not $q$-connected.
Let $1\leq j\leq t$.
By Propositions \ref{Prop_MainThme=1}, \ref{Prop_MainThme=infty}, and \ref{Prop_MainThme<infty} and the definition of $\ind_{n,t}$, we see that $\ind_{n, t}(F_n(M))$ has at least $\sum_{j=1}^t(|\QQ_j|+1)$ irreducible constituents, 
where $F_n$ is the natural equivalence from Theorem \ref{Thm_MoritaEq}. 
By Corollary \ref{MoritaNumbers}, we conclude that $\ind_n(M)$ has at least $\sum_{j=1}^t(|\QQ_j|+1)=r+t$ constituents.
\end{proof}

\begin{remark}
 Note that the lower bound in Theorem \ref{Main_Theorem_Hecke} in general is not sharp.
 Consider for example the case $n=3$, $r=1$, $q=-1$, $Q_1=-1$ over the field $\C$.
 Then $\hh_n$ is isomorphic to the Iwahori-Hecke algebra of type $A_{2}$ with parameter $-1$ and $\hh_{n+1}$ to that of type $A_{3}$ with the same parameter.
 Clearly, the number of $\sim_q$-equivalence classes on $(Q_1)$ is one, so Theorem \ref{Main_Theorem_Hecke} states that $\ind_n(M)$ has at least $2$ constituents for any non-zero $\hh_n$-module $M$. 
 However, one can show that even if $M$ is irreducible, the number of constituents of $\ind_n(M)$ is at least three.
 This is easily shown using the well-known Young rule to compute the induction numbers for the generic Iwahori-Hecke algebras, the decomposition numbers of Iwahori-Hecke algebras of type $A_n$ computed by James in \cite{James}, and the fact that induction commutes with decomposition.
\end{remark}

\begin{remark}
We call any subalgebra $\hh_n'$ of $\hh_n$ generated by a subset of the generators $T_0,\dots, T_{n-1}$ a \emph{parabolic subalgebra of $\hh_n$}.
Adapting the proof of \cite[Thm 1.1]{Schoennenbeck} by using the Mackey formula from \cite{KuMiWa} one can show that $\ind_{\hh_n'}^{\hh_n}(M)$ is reducible for any non-zero $\hh_n'$-module $M$, unless $\hh_n'=\hh_n$. 
However, if $\hh_n'=\hh_{n-1}$, then the statement in Theorem \ref{Main_Theorem_Hecke} is much stronger in general.
\end{remark}

\subsection{Degenerate cyclotomic Hecke algebras}
As already mentioned the Ariki-Koike algebras at $q=1$ are generally not isomorphic to the so-called degenerate cyclotomic Hecke algebras. 
However, a result analogous to Theorem \ref{Main_Theorem_Hecke} still holds:\\
For a non-negative integer $n$ denote by $\hhh_n$ the \emph{degenerate affine Hecke algebra over $K$} as defined by Drinfel'd, cf. \cite{Drinfeld}, i.e. 
as a vector space it is $\hhh_n\cong K[x_1,\dots, x_n]\otimes K[\sy_n]$, the tensor product of the polynomial ring over $K$ in $n$ variables $x_1,\dots, x_n$ and the group algebra over $K$ of the symmetric group $\sy_n$. 
Multiplication is defined such that $K[x_1,\dots, x_n]\otimes 1$ and $1\otimes K[\sy_n]$ are both subalgebras and additionally we have 
\[
 s_ix_j=x_js_i \text{  if } j\neq i,i+1, \quad s_ix_{i+1}=x_is_i+1, \quad x_{i+1}s_i=s_ix_i+1, 
\]
for all sensible values for $i$ and $j$, where $s_i=(i,i+1)$ is the $i$'th standard Coxeter generator of $\sy_n$.
Now let $r\geq 1$ and $\s=(s_1,\dots, s_r)$ in $\widetilde{\Z_{\geq 0}^r}$ .
Then the \emph{degenerate cyclotomic Hecke algebra $\hhh_{n}^{\s}$} is defined as the quotient 
\[
 \hhh_n^{\s}:=\hhh_n/\left\langle(x_1-s_1)\cdots(x_1-s_r)\right\rangle.
\]
The algebra $\hhh_n^{\s}$ embeds into $\hhh_{n+1}^{\s}$ and the corresponding induction functor is exact.

Following Kleshchev (cf. \cite{KleshchevBook}) we can again define refined functors $\iind$ for $0\leq i\leq e:=\text{char}(K)$ and then repeat what we did for $e=\infty$.
An analogue of Proposition \ref{Prop_PropertiesOfiind} holds for $\hhh_n$.
In particular, we once again obtain a crystal graph isomorphism to $B_e(\s)$, cf. \cite[10.3.5]{KleshchevBook} and as in 
Proposition \ref{Prop_MainThme=infty} we obtain the following.
\begin{Theorem}\label{Thm_MainTheoremDegenerate}
 Let $0\neq M\in \hhh_n^{\s}$.
 Then the induced module $\ind_{\hhh_n^{\s}}^{\hhh_{n+1}^{\s}}(M)$ has at least $r+1$ constituents. 
 In particular, it is reducible.
\end{Theorem}

\begin{Remark}
 If $K$ has characteristic $0$, then by \cite[Cor 2]{BrundanKleshchev} the degenerate algebra $\hhh_n^{\s}$ is isomorphic to the Ariki-Koike algebra $\hh_{n, r}(X; X^{s_1},\dots, X^{s_r})$ over $K(X)$, where $X$ is an indeterminate. 
 Hence, in characteristic zero Theorem \ref{Thm_MainTheoremDegenerate} already follows from Proposition \ref{Prop_MainThme=infty}.
\end{Remark}

\section*{Acknowledgments}
This article is a contribution to project I.3 of SFB-TRR 195 ``Symbolic Tools in Mathematics and their Application'' of the German Research Foundation (DFG).\\
I would like to thank my colleague T. Gerber for introducing me to the theory of crystals and a multitude of helpful discussions. 
Along the same line I wish to express my gratitude to my advisor G. Hiss.\\
Finally, I wish to thank the reviewer for a number of helpful comments.
 \bibliographystyle{alpha}
\newcommand{\etalchar}[1]{$^{#1}$}

\end{document}